\documentclass[10pt,twoside]{siamltex}
\usepackage{amsfonts,epsfig}
\usepackage{amssymb}

\usepackage{xy}
\input{xy}
\xyoption{all} \xyoption{rotate}
\newtheorem{remark}[theorem]{Remark}

\newtheorem{example}[theorem]{Example}

\begin{document}

\bibliographystyle{plain}
\title{Nonuniqueness of semidirect decompositions for semidirect products with
directly decomposable factors and applications for dihedral groups}

\author{
Peteris\ Daugulis\thanks{Department of Mathematics, Daugavpils
University, Daugavpils, LV-5400, Latvia (peteris.daugulis@du.lv).
} }

\pagestyle{myheadings} \markboth{P.\ Daugulis}{Nonuniqueness of
semidirect decompositions and applications for dihedral groups}
\maketitle

\begin{abstract} Nonuniqueness of semidirect decompositions of groups is
an insufficiently studied question in contrast to direct
decompositions. We obtain some results about semidirect
decompositions for semidirect products with factors which are
nontrivial direct products. We deal with a special case of
semidirect product when the twisting homomorphism acts diagonally
on a direct product, as well as for the case when the extending
group is a direct product. We give applications of these results
in the case of generalized dihedral groups and classic dihedral
groups $D_{2n}$. For $D_{2n}$ we give a complete description of
semidirect decompositions and values of minimal permutation
degrees.

\end{abstract}

\begin{keywords}
semidirect product, direct product, diagonal action, generalized
dihedral group
\end{keywords}
\begin{AMS}
20E22, 20D40.
\end{AMS}

%%%%%%%%%%%%%%%%%%%%%%%%%%%%%%%%%%%%%%%%%%%%%%%%%%%%%%%%%%%%%
\section{Introduction}\

\subsection{Background} The aim of this article is
to study semidirect decompositions of groups both in general and
special cases.

By the well known Krull-Remak-Schmidt theorem the multiset of
isomorphism types of indecomposable direct factors for groups
satisfying ascending and descending chain conditions on normal
subgroups does not depend on the order of factors. Thus direct
decompositions of such groups, e.g. finite groups, may be
considered understood.

Few results of this type are known for semidirect and Zappa-Szep
decompositions. One can mention the Schur-Zassenhaus theorem as an
example.

We consider cases when the base group or the extending group is a
direct product. We present a general result which allows to
characterize some semidirect decompositions in the case when the
base group is a direct product and the twisting homomorphism acts
diagonally, Prop.\ref{4}.
%We also show that any semidirect product
%is isomorphic to a semidirect product with injective twisting
%homomorphism acting diagonally, Prop.\ref{6}.
Finally we obtain a
nonuniqueness result of semidirect decomposition in the case when
the extending group is a direct product, Prop.\ref{7}. We give
applications of some of these results in the case which is
relatively easy to understand - finite dihedral groups, both
classic and generalized.

We use traditional multiplicative notation for general groups and
additive notation for abelian groups. In this article the dihedral
group of order $m=2n$ is denoted by $D_{m}$ - $D_{m}=\langle a,x|
a^n=e, x^2=e, xax=a^{-1}\rangle$, $m=2n$.  For any $m|n$ we
usually identify $\mathbb{Z}_{m}$ with the corresponding subgroup
of $\mathbb{Z}_{n}$. $Q_{m}$ denotes the dicyclic group of order
$m=4k$ - $Q_{m}=\langle a,x|a^{2k}=e, x^2=a^{k},
x^{-1}ax=a^{-1}\rangle$.

The cyclic group of order $m$ is denoted by $\mathbb{Z}_{m}$, in
additive notation we assume that $\mathbb{Z}_{m}=\langle
1\rangle$. In this article we identify elements of
$\mathbb{Z}_{m}$ and corresponding minimal nonnegative integers.
We use this identification for powers of group elements. For
example, if $r\in \mathbb{Z}_{3}$ and $r\equiv 2 (mod\ 3)$, then
$a^r=a^2$ for any group element $a$.

\subsection{Basic facts about semidirect products}\

We remind the reader that an external semidirect product of groups
$N$ (\sl base group\rm) and $H$ (\sl extending group\rm) is the
group $N\rtimes_{\varphi}H=(N\times H,\cdot)$ where the group
product is defined on the Cartesian product $N\times H$ using a
group homomorphism (\sl twisting homomorphism\rm) $\varphi\in
Hom(H, Aut(N))$ as follows: $(n_{1},h_{1})\cdot
(n_{2},h_{2})=(n_{1}\varphi(h_{1})(n_{2}),h_{1}h_{2})$. Sets
$\widetilde{N}=N\times \{e_{H}\}$ and
$\widetilde{H}=\{e_{N}\}\times H$ are subgroups in $N\times H$.

A group $G$ is an internal semidirect product of its subgroups $N$
and $H$ if $N$ is a normal subgroup, $G=NH$ and $N\cap H=\{e\}$.
If a group $G$ is finite then for $G$ to be an internal semidirect
product $NH$ is equivalent to 1) $N$ being normal in $G$, 2)
$|N|\cdot |H|=|G|$ and 3) $N\cap H=\{e\}$. In the internal case
the twisting homomorphism $H\rightarrow Aut(N)$ is given by the
map $h\mapsto (n\rightarrow hnh^{-1})$, for any $n\in N$, $h\in
H$.

 Both expressions will be called semidirect decompositions of
$G$. If the twisting homomorphism is not discussed, we omit it and
use the notation $\rtimes$. We consider direct product to be a
special case of semidirect product with the twisting homomorphism
being trivial. For relevant treatment see \cite{R1}, \cite{R2}.

A nontrivial semidirect product may admit more than one semidirect
decomposition. Examples are abundant starting from groups of order
$8$.

\begin{example} Twisting homomorphisms are not given in these examples.

$D_{8}\simeq \mathbb{Z}_{4}\rtimes \mathbb{Z}_{2}\simeq
\mathbb{Z}^{2}_{2}\rtimes \mathbb{Z}_{2}$. $\Sigma_{4}\simeq
A_{4}\rtimes \mathbb{Z}_{2}\simeq \mathbb{Z}^{2}_{2}\rtimes
\Sigma_{3}$.

There are semidirect products such that $\mathbb{Z}_{3}\rtimes
Q_{8}\simeq Q_{24}$, but $Q_{8}\rtimes \mathbb{Z}_{3}\simeq
SL(2,\mathbb{F}_{3})$.  On the other hand, there is a group
$G_{32}$ of order $32$, such that $G_{32}\simeq D_{8}\rtimes
\mathbb{Z}^{2}_{2}\simeq \mathbb{Z}^{2}_{2}\rtimes D_{8}$.

Finally, there is a group $G_{24}$ of order $24$ which can be
decomposed in $5$ different ways: $G_{24}\simeq
\mathbb{Z}_{3}\rtimes D_{8}\simeq \mathbb{Z}^{2}_{2}\rtimes
\mathbb{Z}_{3}\simeq D_{12}\rtimes \mathbb{Z}_{2}\simeq
(\mathbb{Z}_{3}\times \mathbb{Z}^{2}_{2})\rtimes
\mathbb{Z}_{2}\simeq Q_{12}\rtimes \mathbb{Z}_{2}$.

\end{example}

\section{Main results}

\subsection{Diagonal semidirect products}

\subsubsection{Automorphisms of direct products}\label{5}

We introduce a linear algebra style notation for direct products
of groups.

Let $G=G_{1}\times G_{2}$. Encode the element $(g_{1},g_{2})$ as a
column $\left[%
\begin{array}{c}
  g_{1} \\
  \hline
  g_{2} \\
\end{array}%
\right] $. If $\varphi\in Aut(G)$, then $\varphi(
\left[%
\begin{array}{c}
  g_{1} \\
  \hline
  g_{2} \\
\end{array}%
\right])=
\left[%
\begin{array}{c}
  \varphi_{1}(g_{1},g_{2}) \\
  \hline
  \varphi_{2}(g_{1},g_{2}) \\
\end{array}%
\right] $. One can theck, that for all relevant parameter values
$\varphi_{i}$ satisfy the following properties:
\begin{enumerate}

\item $\varphi_{i}(ab,e)=\varphi_{i}(a,e)\varphi_{i}(b,e)$,

\item $\varphi_{i}(e,ab)=\varphi_{i}(e,a)\varphi_{i}(e,b)$,

\item
$\varphi_{i}(a,b)=\varphi_{i}(a,e)\varphi_{i}(e,b)=\varphi_{i}(e,b)\varphi_{i}(a,e)$,

\end{enumerate}

Define $\varphi_{11}(g_{1})=\varphi_{1}(g_{1},e)$,
$\varphi_{12}(g_{2})=\varphi_{1}(e,g_{2})$,
$\varphi_{21}(g_{1})=\varphi_{2}(g_{1},e)$,
$\varphi_{22}(g_{2})=\varphi_{2}(e,g_{2})$, for all $g_{i}\in
G_{i}$. All functions $\varphi_{ij}$ are group homomorphisms. Thus
$\varphi_{i}(g_{1},g_{2})=\varphi_{i}(g_{1},e)\varphi_{i}(e,g_{2})=\varphi_{i1}(g_{1})\varphi_{i2}(g_{2})$.

We can encode action of $\varphi$ as follows: $ \varphi(
\left[%
\begin{array}{c}
  g_{1} \\
  \hline
  g_{2} \\
\end{array}%
\right])=
\left[%
\begin{array}{c|c}
  \varphi_{11}(g_{1}) & \varphi_{12}(g_{2}) \\
  \hline
  \varphi_{21}(g_{1}) & \varphi_{22}(g_{2}) \\
\end{array}%
\right]
 $. Thus an automorphism $\varphi\in Aut(G_{1}\times G_{2})$ is
 determined by $4$ group homomorphisms $\varphi_{ij}:G_{j}\rightarrow
 G_{i}$.

\begin{definition}
 We call $\varphi\in Aut(G_{1}\times G_{2})$, $G_{1}\neq \{e\}$, $G_{2}\neq \{e\}$, a \sl diagonal
 automorphism\rm\ if $\varphi_{12}$ and $\varphi_{21}$ are
 trivial homomorphisms.
\end{definition}

\begin{definition}
We call $(G_{1}\times G_{2})\rtimes_{\varphi} H$ a \sl diagonal
semidirect product\rm\ if $\varphi(h)$ is a diagonal $G_{1}\times
G_{2}$-automorphism for any $h\in H$. Explicitly, there are group
homomorphisms $\varphi_{ii}(h):G_{i}\rightarrow G_{i}$ such that
$\varphi(h)(g_{1},g_{2})=(\varphi_{11}(h)(g_{1}),\varphi_{22}(h)(g_{2}))$.
\end{definition}

\begin{remark} $Ker(\varphi)=Ker(\varphi_{11})\cap Ker(\varphi_{22})$.
Note that $G_{i}$ may not be indecomposable as direct factors.
Described encodings and diagonal semidirect products can be
generalized to cases when the base groups splits into an arbitrary
finite number of direct factors. Similar encodings can be used
considering internal semidirect products.

\end{remark}

\subsubsection{Semidirect decompositions of diagonal semidirect products}\

We present a proposition showing nonuniqueness of semidirect
decomposition for diagonal semidirect products. Va\-gue\-ly
speaking, any direct factor of the base group which is invariant
with respect to the initial twisting homomorphism can be moved to
the extending group (nonnormal semidirect factor) to enlarge it.
The new twisting homomorphism is such that the moved direct factor
acts trivially on the remaining part of the base group.

\begin{proposition}\label{4} Let $N_{1},N_{2},H$ be groups. Let $G=(N_{1}\times N_{2})\rtimes_{\varphi}H$
be a diagonal semidirect product,
$\varphi(h)(g_{1},g_{2})=(\varphi_{11}(h)(g_{1}),\varphi_{22}(h)(g_{2}))$.
Then the following statements hold.
\begin{enumerate}

\item $G\simeq N_{1}\rtimes_{\Phi_{11}}
(N_{2}\rtimes_{\varphi_{22}}H)$, for some $\Phi_{11}\in
Hom(N_{2}\rtimes_{\varphi_{22}} H,Aut(N_{1}))$.

\item
$Ker(\Phi_{11})=\widetilde{N_{2}}\widetilde{Ker(\varphi_{11})}$.

\item If $\varphi_{11}(h)=id_{N_{1}}$, for any $h\in H$, i.e.
$\varphi(h)(
\left[%
\begin{array}{c}
  g_{1} \\
  \hline
  g_{2} \\
\end{array}%
\right] )=
\left[%
\begin{array}{c|c}
  g_{1} & e \\
  \hline
  e & \varphi_{22}(h)(g_{2}) \\
\end{array}%
\right] $, then $G\simeq N_{1}\times (N_{2}\rtimes_{\varphi_{22}}
H).$

\end{enumerate}
\end{proposition}

\begin{proof}

1. Consider $N_{1}\rtimes_{\Phi_{11}} (N_{2}\rtimes_{\varphi_{22}}
H)$ where $\Phi_{11}(n_{2},h)=\varphi_{11}(h)$. It is directly
checked that $\Phi_{11}\in Hom(N_{2}\rtimes H,Aut(N_{1}))$. We
will prove that
$$(N_{1}\times N_{2})\rtimes_{\varphi} H \simeq
N_{1}\rtimes_{\Phi_{11}} (N_{2}\rtimes_{\varphi_{22}} H).$$ Define
a bijective map $f:(N_{1}\times N_{2})\rtimes_{\varphi}
H\rightarrow N_{1}\rtimes_{\Phi_{11}} (N_{2}\rtimes_{\varphi_{22}}
H)$ by $f((n_{1},n_{2}),h)=(n_{1},(n_{2},h))$, for all $n_{i}\in
N_{i}$, $h\in H$. We prove that $f$ is a group homomorphism.

Let $a,a'\in N_{1}$, $b,b'\in N_{2}$, $h,h'\in H$. We have that

$((a,b),h)\cdot ((a',b'),h')=((a,b)\varphi(h)(a',b'),hh')=$

$=((a,b)(\varphi_{11}(h)(a'),\varphi_{22}(h)(b')),hh')=((a\varphi_{11}(h)(a'),b\varphi_{22}(h)(b')),hh').$

On the other hand,

$(a,(b,h))\cdot (a',(b',h'))=(a\Phi_{11}(b,h)(a'),(b,h)\cdot
(b',h'))=$

$=(a\varphi_{11}(h)(a'),(b\varphi_{22}(h)(b'),hh'))$. We see that
$f$ is a group isomorphism.

2. $Ker(\Phi_{11})=\{(n_{2},h)|h\in
Ker(\varphi_{11})\}=\widetilde{N_{2}}\widetilde{Ker(\varphi_{11})}$.

3. In notations given above, $\varphi_{11}(h)=id_{N_{1}}$ implies
$\Phi_{11}(n_{2},h)=id_{N_{1}}$, for any $n_{2}\in N_{2}$, $h\in
H$. Thus
%the image of the splitting homomorphism $\Phi_{11}$ of
%$N_{1}\rtimes_{\Phi_{11}} (N_{2}\rtimes_{\varphi_{22}} H)$ is the
%identity and
it is the direct product.
\end{proof}

\begin{example}
Let $G=(\mathbb{Z}_{7}\times
\mathbb{Z}_{9})\rtimes_{\varphi}\mathbb{Z}_{3}$, where
$\varphi(1)(
\left[%
\begin{array}{c}
  g_{1} \\
  \hline
  g_{2} \\
\end{array}%
\right] )=
\left[%
\begin{array}{c|c}
  g^{2}_{1}& e \\
  \hline
  e& g^{4}_{2} \\
\end{array}%
\right]$. In additive notation this can be simplified as follows.
$\varphi(1)(
\left[%
\begin{array}{c}
  g_{1} \\
  \hline
  g_{2} \\
\end{array}%
\right] )=
\left[%
\begin{array}{c|c}
  2g_{1}& 0 \\
  \hline
  0& 4g_{2} \\
\end{array}%
\right]=
\left[%
\begin{array}{c|c}
  2 & 0 \\
  \hline
  0 & 4 \\
\end{array}%
\right]
\left[%
\begin{array}{c}
  g_{1} \\
  \hline
  g_{2} \\
\end{array}%
\right]
 $. $G$ can be defined as the subgroup of $\Sigma_{16}$ generated
 by three permutations:
\begin{enumerate}

\item[a)]
 $(1,...,7)$ (generating $\mathbb{Z}_{7}$),

 \item[b)] $(8,...,16)$ (generating
 $\mathbb{Z}_{9}$) and

 \item[c)]
 $\underbrace{(1,2,4)(3,6,5)}_{\mathbb{Z}_{7}}\underbrace{(8,11,14)(9,15,12)}_{\mathbb{Z}_{9}}$
 (generating action of $\mathbb{Z}_{3}$ on $\mathbb{Z}_{7}\times
 \mathbb{Z}_{9}$).
\end{enumerate}

We have that $G\simeq \mathbb{Z}_{7}\rtimes
 (\mathbb{Z}_{9}\rtimes_{4}
 \mathbb{Z}_{3})\simeq \mathbb{Z}_{9}\rtimes (\mathbb{Z}_{7}\rtimes_{2} \mathbb{Z}_{3})$.

\end{example}

\subsection{Directly decomposable extending groups}\

We show that a direct factor of the extending group can be
transferred to the base group.

\begin{proposition} \label{7} Let $N,H_{1},H_{2}$ be groups Then $$N\rtimes_{\varphi} (H_{1}\times H_{2})\simeq
(N\rtimes_{\varphi_{1}} H_{1})\rtimes_{\varphi_{2}} H_{2},$$ where
$\varphi_{1}(h_{1})(n)=\varphi(h_{1},e_{H_{2}})(n)$ and
$\varphi_{2}(h_{2})(n,h_{1})=(\varphi(e_{H_{1}},h_{2})(n),h_{1})$,
for all $n\in N$, $h_{i}\in H_{i}$.

\end{proposition}

\begin{proof} It is checked that $\varphi_{i}$ are group homomorphisms.

We prove that the map $f:N\rtimes (H_{1}\times
H_{2})\longrightarrow (N\rtimes H_{1})\rtimes H_{2}$ given by
$f(n,(h_{1},h_{2}))=((n,h_{1}),h_{2})$ is a group homomorphism.

Let $n,n'\in N$, $h_{i},h'_{i}\in H_{i}$.

Consider the product $(n,(h_{1},h_{2}))\cdot (n',(h'_{1},h'_{2}))$
in $N\rtimes_{\varphi}(H_{1}\times H_{2})$:

$(n,(h_{1},h_{2}))\cdot
(n',(h'_{1},h'_{2}))=(n\varphi(h_{1},h_{2})(n'),(h_{1}h'_{1},h_{2}h'_{2}))$.

Consider the product $((n,h_{1}),h_{2})\cdot ((n',h'_{1}),h'_{2})$
in $(N\rtimes H_{1})\rtimes H_{2}$:

$ ((n,h_{1}),h_{2})\cdot ((n',h'_{1}),h'_{2})=
(((n,h_{1})\varphi_{2}(h_{2})(n',h'_{1})),h_{2}h'_{2})=$

$=(((n,h_{1})(\varphi(e,h_{2})(n'),h'_{1})),h_{2}h'_{2})=
((n\varphi_{1}(h_{1})(\varphi(e,h_{2})(n')),h_{1}h'_{1}),h_{2}h'_{2})=$

$=((n\varphi(h_{1},h_{2})(n'),h_{1}h'_{1}),h_{2}h'_{2})$.

We see that both products have equal corresponding components and
thus $f$ is a group isomorphism.
\end{proof}

\begin{example} Let $G=\mathbb{Z}_{7}\rtimes_{\varphi} (\mathbb{Z}_{2}\times
\mathbb{Z}_{3})$ where $\varphi(x,y)(1)\equiv (-1)^{x}2^{y} (mod\
7)$. $G$ can be defined as the subgroup of $\Sigma_{7}$ generated
 by three permutations:
\begin{enumerate}

\item[a)]
 $(1,...,7)$ (generating $\mathbb{Z}_{7}$),

 \item[b)] $(1,6)(2,5)(3,4)$ (generating
 action of $\mathbb{Z}_{2}$ on $\mathbb{Z}_{7}$) and

 \item[c)]
 $(1,2,4)(3,6,5)$ (generating
 action of $\mathbb{Z}_{3}$ on $\mathbb{Z}_{7}$).
\end{enumerate}

Then $G\simeq D_{2\cdot 7}\rtimes \mathbb{Z}_{3}\simeq
(\mathbb{Z}_{7}\rtimes \mathbb{Z}_{3})\rtimes \mathbb{Z}_{2}$.

\end{example}
\subsection{Applications}\

\subsubsection{Generalized dihedral groups}\

We remind the reader that an external semidirect product
$D(A)=A\rtimes_{\varphi} \mathbb{Z}_{2}$ is called \sl generalized
 dihedral group\rm\ provided 1) $A$ is abelian and 2) $\varphi(1)(g)=-g$ for any
$g\in A$, in additive notation. We can also denote $D(A)$ by
$A\rtimes_{-1}\mathbb{Z}_{2}$.

Using the well known classification of finite abelian groups we
can assume that $A=\bigoplus_{i=1}^{n}\mathbb{Z}_{m_{i}}$. We use
linear algebra style encoding - we encode $(g_{1},...,g_{n})\in A$
as a column vector $
\left[%
\begin{array}{cc}
  g_{1} \\
  \hline
  ... \\
  \hline
  g_{n} \\
\end{array}%
\right] $. Notations introduced in section \ref{5} are modified
for additive group notation. The action of the twisting
homomorphism is given by scalar or matrix multiplication:
$\varphi(1)(
\left[%
\begin{array}{cc}
  g_{1} \\
  \hline
  ... \\
  \hline
  g_{n} \\
\end{array}%
\right] )=
\left[%
\begin{array}{c|c|c}
  (-g_{1}) & 0 & 0 \\
  \hline
  0 & ... & 0 \\
  \hline
  0 & 0 & (-g_{n}) \\
\end{array}%
\right]= -
\left[%
\begin{array}{cc}
  g_{1} \\
  \hline
  ... \\
  \hline
  g_{n} \\
\end{array}%
\right]=
 (-\textbf{E}_{n})\cdot
\left[%
\begin{array}{cc}
  g_{1} \\
  \hline
  ... \\
  \hline
  g_{n} \\
\end{array}%
\right] $, where $\textbf{E}_{n}$ is the $n\times n$ identity
matrix.

\begin{remark} Generalized dihedral groups are diagonal semidirect
products with an injective twisting homomorphism.

\end{remark}

\begin{proposition} Let $A=\bigoplus_{i=1}^{n}\mathbb{Z}_{m_{i}}$,  let $A=A_{1}\oplus
A_{2}$, where $A_{1}=\bigoplus_{i=1}^{n_{1}}\mathbb{Z}_{m_{i}}$,
$A_{2}=\bigoplus_{i=n_{1}+1}^{n}\mathbb{Z}_{m_{i}}$. Then
$$D(A)\simeq A_{1}\rtimes (A_{2}\rtimes_{-1}\mathbb{Z}_{2})=A_{1}\rtimes D(A_{2})\simeq A_{1}\rtimes D(A/A_{1}).$$

\end{proposition}

\begin{proof}
$D(A)=(A_{1}\oplus A_{2})\rtimes_{\varphi}\mathbb{Z}_{2}$, where
$\varphi(1)(g)=-g$, for any $g\in A$. Thus
$\varphi(g_{1},g_{2})=(-g_{1},-g_{2})$, for any $g_{i}\in G_{i}$.
It follows that $D(A)$ is a diagonal semidirect product with
respect to $A_{1}\oplus A_{2}$ decomposition. According to
Proposition \ref{4} we have that $D(A)\simeq
A_{1}\rtimes_{\Phi_{11}}(A_{2}\rtimes_{\varphi_{22}}
\mathbb{Z}_{2})=A_{1}\rtimes D(A_{2})$, where
$\Phi_{11}(g_{2},1)(g_{1})=\varphi_{11}(1)(g_{1})=-g_{1}$.
\end{proof}

\begin{example}
Let $G=D(\mathbb{Z}_{3}\oplus \mathbb{Z}_{5})$. $G$ can be defined
as a subgroup of $\Sigma_{8}$ generated by permutations $(1,2,3)$,
$(4,5,6,7,8)$ and $(1,2)(4,7)(5,6)$.
  Then $G\simeq
\mathbb{Z}_{3}\rtimes D_{2\cdot 5}\simeq \mathbb{Z}_{5}\rtimes
D_{2\cdot 3}$.
\end{example}

\subsubsection{Dihedral groups}\

Classic dihedral groups are special cases of generalized dihedral
groups when the base group is a cyclic group. We give a complete
description of semidirect decompositions of $D_{2n}$ using both
Proposition \ref{4} and ad hoc computations.

We use a classical presentation of dihedral groups:
$$D_{2n}=\langle a,x| a^n=e, x^2=e, xax=a^{n-1}\rangle = \langle
a\rangle \cup \langle a\rangle x. $$

%= \{e,a,a^2,...,a^{n-1},x,ax,a^2x, ..., a^{n-1}x\}.$$

We note that $D_{2}\simeq \mathbb{Z}_{2}$ and $D_{4}\simeq
\mathbb{Z}_{2}\times \mathbb{Z}_{2}$, in all other cases $D_{2n}$
is nonabelian.

\paragraph{Subgroups}

Let $n\in \mathbb{N}$, $n\ge 3$, $d\in \mathbb{N}$, $d|n$,
$m=\frac{n}{d}$. It is known that $D_{2n}$ has the following
subgroups, see \cite{C}.

\begin{enumerate}

\item For each $m\in \mathbb{N}$ such that $m|n$ there is a
subgroup
$$A_{m}=\langle a^{\frac{n}{m}}\rangle = \langle a^{d} \rangle =
\{e,a^{d},a^{2d},...,a^{(m-1)d}\}\simeq \mathbb{Z}_{m}.$$

$A_{m}\trianglelefteq D_{2n}$ for all $m$. The number of such
subgroups is $d(n)$ (the number of natural $n$-divisors).

\item For each $m\in \mathbb{N}$ such that $m|n$ and each $r\in
\mathbb{Z}_{\frac{n}{m}}=\mathbb{Z}_{d}$ there is a subgroup

$$B_{2m,r}=\langle a^{\frac{n}{m}},a^{r}x\rangle = \langle
a^{d},a^{r}x\rangle =\langle A_{m}, A_{m}(a^rx)\rangle \simeq
D_{2m}.$$

Note that $r\in \mathbb{Z}_{\frac{n}{m}}$ is identified with an
integer as described in the introduction.

The number of such subgroups is $\sigma(n)$ (the sum of natural
$n$-divisors).

 If $2|n$ then $B_{n,r}\trianglelefteq D_{2n}$. In all other
cases, if $1<m<n$ then $B_{2m,r}\not \trianglelefteq D_{2n}$.

\end{enumerate}

\paragraph{Classical decompositions} It known that
$D_{2n}\simeq \mathbb{Z}_{n} \rtimes_{\varphi} \mathbb{Z}_{2}$
where the twisting homomorphism is $\varphi(1)(g)=-g$. In internal
terms, $D_{2n}=A_{n}\rtimes B_{2,r}$, for all $r\in
\mathbb{Z}_{n}$. If $2|n$ and $4\not | n$, then $D_{2n}\simeq
D_{n}\times \mathbb{Z}_{2}$, or, in internal terms,
$D_{2n}=B_{n,r}\times A_{2}$. where $r\in \mathbb{Z}_{2}$. Again,
note, that second indices of $B$-type subgroups can be interpreted
as both integers and residues.

\paragraph{External semidirect decompositions of $D_{2n}$}Using Proposition \ref{4} we get an exaustive description of
external semidirect decompositions of $D_{2n}$.

\begin{proposition} \label{2}

\begin{enumerate}

\item $D_{2n}\simeq \mathbb{Z}_{m}\rtimes_{\varphi}
D_{\frac{2n}{m}}$, for any $m\in \mathbb{N}$, $m|n$, such that
$GCD(m,\frac{n}{m})=1$. $\varphi$ is defined as follows: if
$D_{\frac{2n}{m}}=\langle
a,x|a^{\frac{n}{m}}=e,x^2=e,xax=a^{-1}\rangle$ then
$\varphi(a)(1)=1$ and $\varphi(x)(1)=-1$.

\item $D_{2n}\simeq D_{n}\rtimes_{\varphi} \mathbb{Z}_{2}$, if
$n=2^{\alpha}q$, $\alpha\in \mathbb{N}$.  $\varphi$ is defined as
follows: if $D_{n}=\langle
a,x|a^{\frac{n}{2}}=e,x^2=e,xax=a^{-1}\rangle$ then
$\varphi(1)(a)=a^{-1}$ and $\varphi(1)(x)=ax$.

\item If $2|n$ and $4\not| n$ then $$D_{2n}\simeq D_{n}\times
\mathbb{Z}_{2}.$$

\item There are no other nontrivial external semidirect
decompositions of $D_{2n}$ in the following sense. If
$D_{2n}\simeq X\rtimes Y$, $|X|>1$, $|Y|>1$, then there are two
poosibilities:

\begin{enumerate}

\item[a)] $X=\mathbb{Z}_{m}$ and $Y=D_{\frac{2n}{m}}$, where
$m|n$, $GCD(m,\frac{n}{m})=1$ or

\item[b)] $X=D_{n}$ and $Y=\mathbb{Z}_{2}$, if $2|n$.

\end{enumerate}
\end{enumerate}

\end{proposition}

\begin{proof} Statements 1.,2.,3. are proved by exhibiting a suitable
internal semidirect decomposition.

1. We use the primary decomposition theorem for cyclic groups: if
$n=\prod_{i=1}^{k}p^{\alpha_{i}}_{i}$, then $\mathbb{Z}_{n}\simeq
\bigoplus_{i=1}^{k}\mathbb{Z}_{p^{\alpha_{i}}_{i}}$. The statement
follows from Proposition \ref{4}. Note that $Ker(\varphi)=\langle
a \rangle$.

Alternatively, we prove the same statement using the information
about $D_{2n}$-subgroups. We show that if $GCD(m,\frac{n}{m})=1$
then $D_{2n}=A_{m}\rtimes B_{\frac{2n}{m},r}$.

We have that $A_{m}\trianglelefteq D_{2n}$ and $|A_{m}|\cdot
|B_{\frac{2n}{m},r}|=2n=|D_{2n}|$. $A_{m}\cap B_{\frac{2n}{m}}\le
\langle a^{\frac{n}{m}}\rangle$. Considering subgroups of $\langle
a^{\frac{n}{m}} \rangle$  it follows that $A_{m}\cap
B_{\frac{2n}{m}}=\{e\}$. Thus $D_{2n}=A_{m}\rtimes
B_{\frac{2n}{m},r}\simeq \mathbb{Z}_{m}\rtimes_{\varphi}
D_{\frac{2n}{m}}$. A direct computation shows that $\varphi$ is as
stated: $(a^{m})a^{d}(a^{-m})=a^{d}$,
$(a^{r}x)a^{d}(a^{r}x)=a^{-d}$.

Note that if $2|n$ and $4\not |n$ then $A_{2}\cap B_{n,r}=\{e\}$,
$r\in \mathbb{Z}_{2}$, hence $D_{2n}=A_{2}\times B_{n,r}\simeq
\mathbb{Z}_{2}\times D_{n}$. In this case there are no nontrivial
semidirect decompositions of type $\mathbb{Z}_{2}\rtimes D_{n}$.

2. This case is not covered by Proposition \ref{4}, we show
directly that $D_{2n}=B_{n,0}\rtimes B_{2,1}$.

If $2|n$, then $B_{n,0}\trianglelefteq D_{2n}$. $|B_{n,0}|\cdot
|B_{2,1}|=|D_{2n}|$. It can be checked that $B_{n,0}\cap
B_{2,1}=\{e\}$: $B_{n,0}=\langle a^2, x\rangle$, $B_{2,1}=\langle
ax \rangle$.

Thus $D_{2n}\simeq D_{n}\rtimes \mathbb{Z}_{2}$. A direct
computation shows that $\varphi$ is as stated: $(ax)a^2$ $(ax)=$
$a^{-2}$ (the generator $a^2$ gets inverted), $(ax)x(ax)=a^2x$
(the generator $x$ gets multiplied by the other generator $a^2$).

3. Using Proposition \ref{4} we see that
$D_{2n}=D(\mathbb{Z}_{n})=(\mathbb{Z}_{2}\oplus ...)\rtimes
\mathbb{Z}_{2}\simeq \mathbb{Z}_{2}\times
D(\mathbb{Z}_{n}/\mathbb{Z}_{2})\simeq \mathbb{Z}_{2}\times
D_{n}$.

It can also be proved using the list of subgroups. We remind that
$D_{2n}=B_{n,0}\times A_{2}\simeq D_{n}\times \mathbb{Z}_{2}$ for
the following reasons. Both subgroups are normal. $|B_{n,0}|\cdot
|A_{2}|=|D_{2n}|$. $B_{n,0}=\langle a^2, x\rangle$, $A_{2}=\langle
a^{\frac{n}{2}}\rangle$, $\frac{n}{2}$ is odd, therefore
$B_{n,0}\cap A_{2}=\{e\}$.

4. Consider all possible internal semidirect decompositions of
$D_{2n}$.

If $D_{2n}=X\rtimes Y$ then $X$ must be a normal subgroup of
$D_{2n}$ therefore $X$ must be $A_{m}$ or $B_{n,r}$ with $2|n$.

If $X=A_{m}$ then $Y$ must be $B_{m',r}$ ir order to generate
$D_{2n}$, with $m'=\frac{2n}{m}$. $A_{m}\cap
B_{\frac{2n}{m},r}=\{e\}$ iff $GCD(n,\frac{n}{m})=1$.

Let $X=B_{n,r}$ with $2|n$, $r\in \mathbb{Z}_{2}$. There are $n+1$
subgroups of $D_{2n}$ having order $2$: $B_{2,r}\ $, $r\in
\mathbb{Z}_{n}$ and $A_2=\langle a^{\frac{n}{2}}\rangle$. For any
$n$ such that $2|n$ this gives a semidirect decomposition of type
$D_{n}\rtimes \mathbb{Z}_{2}$. If $4\not |n$ then $A_{2}\cap
B_{n,r}=\{e\}$ which gives a direct decomposition $D_{n}\times
\mathbb{Z}_2$.
\end{proof}

%\begin{lemma} \label{1} Let $n,a,b\in \mathbb{N}$, $a|n$, $b|n$. Then $$|\mathbb{Z}_{a}\cap
%\mathbb{Z}_{b}|>1\ iff\ GCD(a,b)>1.$$
%
%\end{lemma}
%
%\begin{proof} As it was noted above, if $m|n$ then we identify $\mathbb{Z}_{m}$ with the appopriate
%subgroup of $\mathbb{Z}_{n}$. Direct computations show that $GCD(a,b)=d$ iff
%$LCM(\frac{n}{a},\frac{n}{b})=\frac{n}{d}$.
%
%Using the multiplicative notation of cyclic groups we can assume
%that $\mathbb{Z}_{n}=\langle z|z^n=e \rangle$,
%$\mathbb{Z}_{a}=\langle z^{\frac{n}{a}} \rangle$,
%$\mathbb{Z}_{b}=\langle z^{\frac{n}{b}} \rangle$.
%$|\mathbb{Z}_{a}\cap \mathbb{Z}_{b}|>1$ iff
%$LCM(\frac{n}{a},\frac{n}{b})<n$ iff $GCD(a,b)>1$.
%
%\end{proof}

\begin{remark} In terms of prime factorization the condition $GCD(m,\frac{n}{m})=1$ is
equivalent to the fact that $m$ and $\frac{n}{m}$ are products of
full prime powers of the prime factorization of $n$. Existence of
many members of this family also follows from Schur-Zassenhaus
theorem. If $m|n$ and $GCD(m,\frac{n}{m})=1$ then
$GCD(|A_{m}|,|D_{2n}/A_{m}|)=1$, $D_{2n}/A_{m}\simeq
D_{\frac{2n}{m}}$ and, hence $D_{2n}\simeq A_{m}\rtimes
D_{\frac{2n}{m}}$.

\end{remark}

\begin{remark} Note that there are at most $2$ external semidirect decompositions when $n$ is a
prime power:
\begin{enumerate}

\item if $n=p^{\alpha}$, $p$ an odd prime, then there is only one
(classical) external semidirect decomposition:
$D_{2p^{\alpha}}\simeq \mathbb{Z}_{p^{\alpha}}\rtimes
\mathbb{Z}_{2}$,

\item if $n=2^{\alpha}$, $\alpha\ge 3$, then there are two
external semidirect decompositions: $D_{2\cdot 2^{\alpha}}\simeq
\mathbb{Z}_{2^{\alpha}}\rtimes \mathbb{Z}_{2}\simeq
D_{2^{\alpha}}\rtimes \mathbb{Z}_{2}$.

\end{enumerate}

\end{remark}

\begin{remark}
The image of the twisting homomorphism in each case of a proper
semidirect product is isomorphic to $\mathbb{Z}_{2}$. If the
extending group is not $\mathbb{Z}_{2}$, then the twisting
homomorphism is not injective.

\end{remark}

\begin{example}

%We give lists of external semidirect decompositions for some $n$.

%External semidirect decompositions of $D_{90}$:
%
%$D_{90}\simeq \mathbb{Z}_{45}\rtimes \mathbb{Z}_{2}\simeq
%\mathbb{Z}_{9}\rtimes D_{10}\simeq \mathbb{Z}_{5}\rtimes D_{18}$.
%
%External semidirect decompositions of $D_{120}$:
%
%$D_{120}\simeq \mathbb{Z}_{60}\rtimes \mathbb{Z}_{2}\simeq
%\mathbb{Z}_{12}\rtimes D_{10}\simeq \mathbb{Z}_{20}\rtimes
%D_{6}\simeq \mathbb{Z}_{15}\rtimes
%\mathbb{D}_{8}\simeq\mathbb{Z}_{4}\rtimes
%\mathbb{D}_{30}\simeq\mathbb{Z}_{3}\rtimes D_{40}\simeq$
%
%$\simeq\mathbb{Z}_{5}\rtimes D_{24}\simeq D_{60}\rtimes
%\mathbb{Z}_{2}$.

External semidirect decompositions of $D_{2\cdot 30}$:

$D_{60}\simeq \mathbb{Z}_{30}\rtimes \mathbb{Z}_{2}\simeq
\mathbb{Z}_{6}\rtimes D_{10}\simeq \mathbb{Z}_{10}\rtimes
D_{6}\simeq\mathbb{Z}_{15}\rtimes D_{4}\simeq\mathbb{Z}_{3}\rtimes
D_{20}\simeq \mathbb{Z}_{5}\rtimes D_{12}\simeq$

$\simeq D_{30} \rtimes \mathbb{Z}_{2}\simeq D_{30}\times
\mathbb{Z}_{2}$.

\end{example}

%\subsection{Internal view of semidirect decompositions}\

\paragraph{Internal semidirect decompositions of $D_{2n}$} We now describe all internal semidirect decompositions of
$D_{2n}$.
\newpage

\begin{proposition} Let $n\in \mathbb{N}$.

\begin{enumerate}

\item If $m\in \mathbb{N}$, $m|n$, is such that
$GCD(m,\frac{n}{m})=1$, then $$D_{2n} = A_{m}\rtimes
B_{\frac{2n}{m},r},$$ for all $r\in \mathbb{Z}_{m}$.

\item If $n=2^{\alpha}q$, $\alpha\in \mathbb{N}$, then
$$D_{2n}=B_{n,0}\rtimes B_{2,r_{1}}=B_{n,1}\rtimes B_{2,r_{0}}$$
where $r_{i}\in \mathbb{Z}_{n}$, $r_{i}\equiv i(mod\ 2)$.

\item If $2|n$ and $4\not n$ then $D_{2n}=B_{n,0}\times A_{2}$ and
$D_{2n}=B_{n,1}\times A_{2}$.

\item There are no other internal semidirect decompositions of
$D_{2n}$.

\end{enumerate}

\end{proposition}

\begin{proof}

1. We look for internal semidirect decompositions of $D_{2n}$ in
form $A_{m}\rtimes B_{m',r}$. We must have $m'=\frac{2n}{m}$ and
$r\in \mathbb{Z}_{m}$. $A_{m}\cap B_{\frac{2n}{m},r}=\{e\}$ iff
$GCD(m,\frac{n}{m})=1$. Thus $D_{2n}=A_{m}\rtimes
B_{\frac{2n}{m},r}$ for all $m$ such that $GCD(m,\frac{n}{m})=1$
and all $r\in \mathbb{Z}_{m}$ are the only possible decompositions
of this kind.

2. We look for internal semidirect decompositions of $D_{2n}$ in
form $B_{m,r}\rtimes B_{m',r'}$. We must have
$B_{m,r}\trianglelefteq D_{2n}$ therefore $2|n$, $m=n$ and $r\in
\mathbb{Z}_{2}$, thus we have two possible decomposition series:
$B_{n,0}\rtimes B_{2,r'}$ and $B_{n,1}\rtimes B_{2,r''}$. To
ensure trivial intersections of semidirect factors we must have
$r'\equiv 1(mod\ 2)$ and $r''\equiv 0(mod\ 2)$.

3. If $2|n$ and $4\not n$ then $B_{n,0}\cap A_{2}=B_{n,1}\cap
A_{2}=\{e\}$ where all subgroups are normal.

4. It follows from the previous arguments.
\end{proof}

\paragraph{Permutation representations of dihedral groups}

Finally we find minimal degrees of faithful permutation
representations of $D_{2n}$. If $n$ is not a prime power then
these numbers are smaller than degrees of classical permutation
representations of dihedral groups. This is a consequence of
\ref{2} and Karpilovsky bounds for finite abelian groups \cite{K}.

Let $\mu(G)$ be the minimal faithful permutation representation
degree of $G$, i.e. the minimal $n\in \mathbb{N}$ such that there
is an injective group homomorphism $G\rightarrow \Sigma_{n}$. It
is known that for finite groups $G,H$ and a group homomorphism
$\varphi:H\rightarrow Aut(G)$ we have that $\mu(G\rtimes_{\varphi}
H)\le |G|+\mu(H)$. If, additionally, $\varphi$ is injective, then
$\mu(G\rtimes_{\varphi} H)\le |G|$.

\begin{proposition} Let $n=\prod_{i}p_{i}^{\alpha_{i}}$ be the prime
factorization of $n\in \mathbb{N}$. Then $\mu(D_{2n})=
\sum_{i}p_{i}^{\alpha_{i}}$.

\end{proposition}

\begin{proof} First we prove that
\begin{equation} \label{8} \mu(D_{2n})\le \sum_{i}p_{i}^{\alpha_{i}}.
\end{equation}
By statement 1. of \ref{2} we have that $D_{2n}\simeq
\mathbb{Z}_{p_{1}^{\alpha_{1}}}\rtimes D_{2n_{1}}$, where
$n_{1}=\frac{n}{p_{1}^{\alpha_{1}}}$. Thus $\mu(D_{2n})\le
p_{1}^{\alpha_{1}}+\mu(D_{2n_{1}})$. \ref{8} follows by induction
in $i$ using injectivity of the twisting homomorphism at the last
step.

To prove the opposite inequality and the statement, we note that
$D_{2n}\simeq \mathbb{Z}_{n}\rtimes \mathbb{Z}_{2}$, thus
$\mathbb{Z}_{n}\le D_{2n}$. It implies $\mu(\mathbb{Z}_{n})\le
\mu(D_{2n})$, therefore $\sum_{i}p_{i}^{\alpha_{i}}\le
\mu(D_{2n})$ by the Karpilovsky theorem for abelian groups
\cite{K}.
\end{proof}

\begin{example} $\min_{n\in \mathbb{N}}\{n: \mu(D_{2n})<n\}=6$: $\mu(D_{2\cdot 6})=5$, $D_{2\cdot 6}$ can be generated by $(1,2,3)$, $(1,2)$, $(4,5)$.
If, additionally, $D_{2n}$ is directly indecomposable, then the
minimum is $12$: $\mu(D_{2\cdot 12})=7$, $D_{2\cdot 12}$ can be
generated by $(1,2,3,4)$, $(5,6,7)$, $(1,3)(5,6)$.

\end{example}

\section{Conclusion} We have obtained results showing possibility of
various semidirect decompositions of a given semidirect product in
two cases: 1) if the original twisting homomorphism is diagonal
and the base group is directly decomposable and 2) if the
extending group is directly decomposable. These results may
stimulate further interest in looking for analogues of
Krull-Remak-Schmidt theorem type results for semidirect and
Zappa-Szep products.

We have presented semidirect decompositions of generalized
dihedral groups and classical dihedral groups as an application.
Apart from semidirect decompositions guarranteed by the general
proposition \ref{4}, for $D_{2n}$ there are additional
decompositions of external type $D_{n}\rtimes \mathbb{Z}_{2}$ if
$2|n$.

Semidirect decompositions of dihedral groups give the exact value
of $\mu(D_{2n})$.

\section*{Acknowledgement} Computations were performed using the
computational algebra system MAGMA, see Bosma et al. \cite{B}.

%%%%%%%%%%%%%%%%%%%%%%%%%%%%%%%%%%%%%%%%%%%%%%%%%%%%%%%%%%%%%

\end{document}